\documentclass{amsart}
  \usepackage{amsmath,amsfonts}
\usepackage{fixltx2e,pmat}
\usepackage{subfigure,graphicx}

\usepackage[breaklinks]{hyperref}
\usepackage{booktabs,fixltx2e,fullpage,multicol}

\numberwithin{equation}{section}
\numberwithin{table}{section}

\theoremstyle{plain}
\newtheorem{thm}{Theorem}[section]
\newtheorem{defn}[thm]{Definition}
\newtheorem{ex}[thm]{Example}
\newtheorem{corollary}[thm]{Corollary}

\newtheorem{theorem}{Theorem}[section]

\theoremstyle{remark}
\newtheorem{remark}[theorem]{Remark}

\DeclareMathOperator{\rank}{rank}

\newcommand{\normdist}{\mathcal{N}}
\newcommand{\range}{\mathcal{R}}
\newcommand{\trans}{\mathsf{T}}

\numberwithin{equation}{section}
\numberwithin{table}{section}

\DeclareMathOperator{\cov}{cov}

\newcommand{\post}{\text{post}}
\newcommand{\prior}{\text{prior}}

\title{Differential algebra for model comparison}
\author{Heather~A.~Harrington, %\footnotemark[2]\ 
 Kenneth~L.~Ho, %\footnotemark[3]\ 
 Nicolette~Meshkat}
\address{Heather A Harrington: Mathematical Institute, University of Oxford, Oxford OX2 6GG, UK, Kenneth~L.~Ho: Department of Mathematics, Stanford University, Stanford, CA 94305, USA, Nicolette~Meshkat: Department of Mathematics, Box 8205, North Carolina State University, Raleigh, NC, 27695-8205, USA and Department of Mathematics and Computer Science, 500 El Camino Real, Santa Clara University, Santa Clara, CA, 95053}

\begin{document}
\maketitle
%\slugger{siap}{xxxx}{xx}{x}{x--x}%slugger should be set to mms, siap, sicomp, sicon, sidma, sima, simax, sinum, siopt, sisc, or sirev

%\renewcommand{\thefootnote}{\fnsymbol{footnote}}
%\footnotetext[2]{Mathematical Institute, University of Oxford, Oxford OX2 6GG, UK}
%\footnotetext[3]{Department of Mathematics, Stanford University, Stanford, CA 94305, USA}
%\footnotetext[4]{Department of Mathematics, Box 8205, North Carolina State University, Raleigh, NC, 27695-8205, USA}
%\footnotetext[5]{Department of Mathematics, Box 8205, North Carolina State University, Raleigh, NC, 27695-8205, USA}
%
%\renewcommand{\thefootnote}{\arabic{footnote}}

\begin{abstract}
We present a method for rejecting competing models from noisy time-course data that does not rely on parameter inference. First we characterize ordinary differential equation models in only measurable variables using differential algebra elimination. Next we extract additional information from the given data using Gaussian Process Regression (GPR) and then transform the differential invariants. We develop a test using linear algebra and statistics to reject transformed models with the given data in a parameter-free manner. This algorithm exploits the information about transients that is encoded in the model's structure. We demonstrate the power of this approach by discriminating between different models from mathematical biology.
%
%\begin{enumerate}
%\item Algorithm for rejecting models from transient data-- include noise 
%\item Results: linear compartment models, periodic (L-V, 3 species L-V), Lorenz
%\end{enumerate}
%Differential invariants are important and can do this in a parameter-free framework. Present novel procedure...
%Results
%\begin{itemize}
%\item Given data (output) and known input, discriminate between different numbers of compartments for LCM with exact I/O.
%\item Guide experimental design: can determine input necessary for identifiable parameters for LCM (persistent excitation)
%\item With enough data (even with noise) can determine whether a solution exists
%
%Here we develop parameter-free methodologies for comparing models and time-course data without parameter inference. First we introduce a new way to characterise such ordinary differential equation models for model comparison and test them with time-course data by employing techniques from computational differential algebra, Gaussian processes, linear algebra and statistics.
%
%\end{itemize}
\end{abstract}

%\begin{keywords} Model selection, differential algebra, algebraic statistics, mathematical biology \end{keywords}
keywords: Model selection, differential algebra, algebraic statistics, mathematical biology

%\begin{AMS} 34A09, 62-07, 62P10, 65D25, 92B05 \end{AMS}

%\pagestyle{myheadings}
\thispagestyle{plain}
%\markboth{TEX PRODUCTION}{USING SIAM'S \LaTeX\ MACROS}

\section{Introduction}

%\begin{itemize}
%\item what's been done\\

Given competing mathematical models to describe a process, we wish to know whether our data is compatible with the candidate models. Often comparing models requires optimization and fitting time course data to estimate parameter values and then applying an information criterion to select a `best' model \cite{akaike}. However sometimes it is not feasible to estimate the value of these unknown parameters (e.g. large parameter space, nonlinear objective function, nonidentifiable etc). 

The parameter problem has motivated the growth of fields that embrace a parameter-free flavour such as chemical reaction network theory and stoichiometric theory \cite{feinberg:1987:chem-eng-sci,feinberg:1988:chem-eng-sci,clarke:1988:cell-biophys}.  
%should we be citing flux balance analysis too?
However many of these approaches are limited to comparing the behavior of models at steady-state \cite{gunawardena:2007:biophys-j,manrai:2008:biophys-j,conradi:2005:iee-proc-syst-biol}. Inspired by techniques commonly used in applied algebraic geometry \cite{Cox} and algebraic statistics \cite{algstat}, methods for discriminating between models without estimating parameters has been developed for steady-state data \cite{harrington-pnas-2012}, applied to models in Wnt signaling \cite{maclean2015,gross:2016:bmb}, and then generalized to only include one data point \cite{gross:2016:bmb,gross:2015:NAG}. Briefly, these approaches characterize a model $f(x_1,\ldots,x_N, p_1, \ldots p_R)$ in only observable variables $g(x_1,\ldots,x_M, p_1, \ldots p_R)$ using techniques from computational algebraic geometry and tests whether the steady-state data are coplanar with this new characterization of the model, called a {\it steady-state invariant} \cite{gunawardena:2007:biophys-j}. Notably the method doesn't require parameter estimation, and also includes a statistical cut-off for model compatibility with noisy data.

Here, we present a method for comparing models with {\it time course data} via computing a {\it differential invariant}.
We consider models of the form $\bf{\dot{x}}(t) = \textbf{f}(\textbf{x}(t),\textbf{u}(t),\textbf{p})$ and $\textbf{y}(t)=\textbf{g}(\textbf{x}(t),\textbf{p})$ where $u_i(t)$ is a known input into the system, $i=1,...,L$, $y_j(t)$ is a known output (measurement) from the system, $j=1,...,M$, $x_k(t)$ are species variables, $k=1,...,N$, $\textbf{p}$ is the unknown $R-$dimensional parameter vector, and the functions $\textbf{f},\textbf{g}$ are rational functions of their arguments.  The dynamics of the model can be observed in terms of a time series where $\textbf{u}(t)$ is the input at discrete points and $\textbf{y}(t)$ is the output.

In this setting, we aim to characterize our ODE models by eliminating variables we cannot measure using differential elimination from differential algebra. From the elimination, we form a differential invariant, where the differential monomials have coefficients that are functions of the parameters $p_1,\ldots, p_R$.  We obtain a system of equations in 0,1, and higher order derivatives and we write this implicit system of equations as $F_j(\textbf{u}, \bf{\dot{u}},\bf{\ddot{u}},\bf{\dddot{u}},\ldots, \textbf{y}, \bf{\dot{y}},\bf{\ddot{y}},\bf{\dddot{y}},\ldots)=0$, $j=1,...,M$, and call these the input-output equations our {\it differential invariants}.
Specifically, we have equations of the form:
$$ \sum_{i}{c_{i}(\textbf{p})\psi_{i}(\textbf{u},\textbf{y})}=0 $$
where $c_{i}(\textbf{p})$ are rational functions of the parameters and $\psi_{i}(\textbf{u},\textbf{y})$ are differential monomials, i.e. monomials in $\textbf{u}, \bf{\dot{u}},\bf{\ddot{u}},\bf{\dddot{u}},\ldots, \textbf{y}, \bf{\dot{y}},\bf{\ddot{y}},\bf{\dddot{y}},\ldots$.  We will see shortly that in the linear case, $F_j$ is a linear differential equation. For non-linear models, $F_j$ is nonlinear.  

If we substitute into the differential invariant available data into the observable monomials for each of the time points, we can form a linear system of equations (each row is a different time point).  Then we ask: does there exist a $\kappa$ such that $A \kappa = b$. If $b = 0$ of course we are guaranteed a zero trivial solution and the non-trivial case can be determined via a rank test (i.e., SVD) and can perform the statistical criterion developed in \cite{harrington-pnas-2012} with the bound improved in \cite{maclean2015}, but for $A \kappa = b$ there may be no solutions.  Thus, we must check if the linear system of equations $A \kappa = b$ is consistent, i.e. has one or infinitely many solutions. Assuming measurement noise is known, we derive a statistical cut-off for when the model is incompatible with the data.

However suppose that one does not have data points for the higher order derivative data, then these need to be estimated. We present a method using Gaussian Process Regression (GPR) to estimate the time course data using a GPR. Since the derivative of a GP is also GP, so we can estimate the higher order derivative of the data as well as the measurement noise introduced and estimate the error introduced during the GPR (so we can discard points with too much GPR estimation error). % (decided by the hyper parameters in our maximum likelihood estimation) [Ken is this correct?!]
This enables us to input derivative data into the differential invariant and test model compatibility using the solvability test with the statistical cut-off we present.

We showcase our method throughout with examples from linear and nonlinear models.

\par

\section{Differential Elimination}

We now give some background on differential algebra since a crucial step in our algorithm is to perform differential elimination to obtain equations purely in terms of input variables, output variables, and parameters.  For this reason, we will only give background on the ideas from differential algebra required to understand the differential elimination process.  For a more detailed description of differential algebra and the algorithms listed below, see \cite{Aistleitner,Kolchin,Ritt}.  In what follows, we assume the reader is familiar with concepts such as \textit{rings} and \textit{ideals}, which are covered in great detail in \cite{Cox}.

\begin{defn} A ring $S$ is said to be a \textit{differential ring} if there is a derivative defined on $S$ and $S$ is closed under differentiation.  A \textit{differential ideal} is an ideal which is closed under differentiation.
\end{defn}

A useful description of a differential ideal is called a \textit{differential characteristic set}, which is a finite description of a possibly infinite set of differential polynomials.  We give the technical definition from \cite{Ritt}:

\begin{defn} Let $\Sigma$ be a set of differential polynomials, not necessarily finite.  If $A \subset \Sigma$ is an auto-reduced set, such that no lower ranked auto-reduced set can be formed in $\Sigma$, then $A$ is called a \textit{differential characteristic set}.
\end{defn}

A well-known fact in differential algebra is that differential ideals need not be finitely generated \cite{Kolchin,Ritt}.  However, a radical differential ideal is finitely generated by the \textit{Ritt-Raudenbush basis theorem} \cite{Kaplansky}.  This result gives rise to Ritt's pseudodivision algorithm (see below), allowing us to compute the differential characteristic set of a radical differential ideal.  We now describe various methods to find a differential characteristic set and other related notions, and we describe why they are relevant to our problem, namely, they can be used to find the \textit{input-output equations}.

Consider an ODE system of the form $\bf{\dot{x}}(t) = \textbf{f}(\textbf{x}(t),\textbf{p},\textbf{u}(t))$ and $y_j(t)=g_j(\textbf{x}(t),\textbf{p})$ for $j=1,...,M$ with $\textbf{f}$ and $\textbf{g}$ rational functions of their arguments.  Let our differential ideal be generated by the differential polynomials obtained by subtracting the right-hand-side from the ODE system to obtain $\bf{\dot{x}}(t) - \textbf{f}(\textbf{x}(t),\textbf{p},\textbf{u}(t))$ and $y_j(t) - g_j(\textbf{x}(t),\textbf{p})$ for $j=1,...,M$.  Then a differential characteristic set is of the form \cite{Saccomani}:
\begin{align*}
A_1(\textbf{u},\textbf{y})&,...,A_M(\textbf{u},\textbf{y})\\
A_{M+1}&(\textbf{u},\textbf{y},x_1)\\
A_{M+2}&(\textbf{u},\textbf{y},x_1,x_2)\\
&...\\
A_{M+N}&(\textbf{u},\textbf{y},x_1,...,x_N)
\end{align*}
The first $M$ terms of the differential characteristic set, $A_1(\textbf{u},\textbf{y}),...,A_M(\textbf{u},\textbf{y})$, are those terms independent of the state variables and when set to zero form the \textit{input-output equations}:
 $$\textbf{F}(\textbf{u}, \bf{\dot{u}}, \bf{\ddot{u}}, \bf{\dddot{u}},\ldots, \textbf{y}, \bf{\dot{y}}, \bf{\ddot{y}}, \bf{\dddot{y}},\ldots)=\textbf{0}.$$  
Specifically, the $M$ input-output equations $\textbf{F}(\textbf{u}, \bf{\dot{u}}, \bf{\ddot{u}}, \bf{\dddot{u}},\ldots, \textbf{y}, \bf{\dot{y}}, \bf{\ddot{y}}, \bf{\dddot{y}},\ldots)=\textbf{0}$ are polynomial equations in the variables $\textbf{u}, \bf{\dot{u}}, \bf{\ddot{u}}, \bf{\dddot{u}},\ldots, \textbf{y}, \bf{\dot{y}}, \bf{\ddot{y}}, \bf{\dddot{y}},\ldots$ with rational coefficients in the parameter vector $\textbf{p}$.  Note that the differential characteristic set is in general non-unique, but the coefficients of the input-output equations can be fixed uniquely by normalizing the equations to make them monic.

We now discuss several methods to find the input-output equations.  The first method (Ritt's pseudodivision algorithm) can be used to find a differential characteristic set for a radical differential ideal.  The second method (RosenfeldGroebner) gives a representation of the radical of the differential ideal as an intersection of regular differential ideals and can also be used to find a differential characteristic set under certain conditions \cite{Boulier2, Golubitsky}.  Finally, we discuss Gr\"{o}bner basis methods to find the \textit{input-output equations}.

\subsection{Ritt's pseudodivision algorithm}

A differential characteristic set of a prime differential ideal is a set of generators for the ideal \cite{Forsman}.  An algorithm to find a differential characteristic set of a radical (in particular, prime) differential ideal generated by a finite set of differential polynomals is called Ritt's pseudodivision algorithm.  We describe the process in detail below, which comes from the description in \cite{Saccomani}.  Note that our differential ideal as described above is a prime differential ideal \cite{Diop,Ritt}.

Let $u_j$ be the leader of a polynomial $A_j$, which is the highest ranking derivative of the variables appearing in that polynomial.  A polynomial $A_i$ is said to be of \textit{lower rank} than $A_j$ if $u_i < u_j$ or, whenever $u_i = u_j$, the algebraic degree of the leader of $A_i$ is less than the algebraic degree of the leader of $A_j$.  A polynomial $A_i$ is \textit{reduced with respect to a polynomial} $A_j$ if $A_i$ contains neither the leader of $A_j$ with equal or greater algebraic degree, nor its derivatives.  If $A_i$ is not reduced with respect to $A_j$, it can be reduced by using the pseudodivision algorithm below.    

\begin{enumerate}
\item If $A_i$ contains the $k^{th}$ derivative $u_j^{(k)}$ of the leader of $A_j$, $A_j$ is differentiated $k$ times so its leader becomes $u_j^{(k)}$.
\item Multiply the polynomial $A_i$ by the coefficient of the highest power of $u_j^{(k)}$; let $R$ be the remainder of the division of this new polynomial by $A_j^{(k)}$ with respect to the variable $u_j^{(k)}$.  Then $R$ is reduced with respect to $A_j^{(k)}$.  The polynomial $R$ is called the \textit{pseudoremainder} of the pseudodivision.
\item The polynomial $A_i$ is replaced by the pseudoremainder $R$ and the process is iterated using $A_j^{(k-1)}$ in place of $A_j^{(k)}$ and so on, until the pseudoremainder is reduced with respect to $A_j$.
\end{enumerate}

This algorithm is applied to a set of differential polynomials, such that each polynomial is reduced with respect to each other, to form an auto-reduced set.  The result is a differential characteristic set.

\subsection{RosenfeldGroebner}

Using the {\tt DifferentialAlgebra} package in Maple, one can find a representation of the radical of a differential ideal generated by some equations, as an intersection of radical differential ideals with respect to a given ranking and rewrites a prime differential ideal using a different ranking \cite{Maple}.  Specifically, the {\tt RosenfeldGroebner} command in Maple takes two arguments: {\tt sys} and {\tt R}, where {\tt sys} is a list of set of differential equations or inequations which are all rational in the independent and dependent variables and their derivatives and {\tt R} is a differential polynomial ring built by the command {\tt DifferentialRing} specifying the independent and dependent variables and a ranking for them \cite{Maple}.  Then {\tt RosenfeldGroebner} returns a representation of the radical of the differential ideal generated by {\tt sys}, as an intersection of radical differential ideals saturated by the multiplicative family generated by the inequations found in {\tt sys}.  This representation consists of a list of regular differential chains with respect to the ranking of {\tt R}.  Note that {\tt RosenfeldGroebner} returns a differential characteristic set if the differential ideal is prime \cite{Boulier2}.

\subsection{Gr\"{o}bner basis methods}

Finally, both algebraic and differential Gr\"{o}bner bases can be employed to find the input-output equations.  To use an algebraic Gr\"{o}bner basis, one can take a sufficient number of derivatives of the model equations and then treat the derivatives of the variables as indeterminates in the polynomial ring in $\textbf{x}$, $\bf{\dot{x}}$, $\bf{\ddot{x}}$,..., $\textbf{u}$, $\bf{\dot{u}}$, $\bf{\ddot{u}}$,..., $\textbf{y}$, $\bf{\dot{y}}$, $\bf{\ddot{y}}$,..., etc.  Then a Gr\"{o}bner basis of the ideal generated by this full system of (differential) equations with an elimination ordering where the state variables and their derivatives are eliminated first can be found.  Details of this approach can be found in \cite{Meshkat2012}.  Differential Gr\"{o}bner bases have been developed by Carr\`{a} Ferro \cite{CarraFerro}, Ollivier \cite{Ollivier2}, and Mansfield \cite{Mansfield}, but currently there are no implementations in computer algebra systems \cite{Aistleitner}.

\subsection{Model rejection using differential invariants}

We now discuss how to use the differential invariants obtained from differential elimination (using Ritt's pseudodivision, differential Groebner bases, or some other method) for model selection/rejection.

Recall our input-output relations, or differential invariants, are of the form:
$$ \sum_{i}{c_{i}(\textbf{p})\psi_{i}(\textbf{u},\textbf{y})}=0 $$
The functions $\psi_{i}(\textbf{u},\textbf{y})$ are differential monomials, i.e. monomials in the input/output variables $\textbf{u}, \bf{\dot{u}},\bf{\ddot{u}},\bf{\dddot{u}}$, $\ldots$, $\textbf{y}, \bf{\dot{y}},\bf{\ddot{y}},\bf{\dddot{y}},\ldots$, etc, and the functions $c_{i}(\textbf{p})$ are rational functions in the unknown parameter vector $\textbf{p}$.  In order to uniquely fix the rational coefficients $c_{i}(\textbf{p})$ to the differential monomials $\psi_{i}(\textbf{u},\textbf{y})$, we normalize each input/output equation to make it monic.  In other words, we can re-write our input-output relations as:
$$ \sum_{i}{\tilde{c}_{i}(\textbf{p})\psi_{i}(\textbf{u},\textbf{y})}= \xi (\textbf{u},\textbf{y}) $$
Here $\xi (\textbf{u},\textbf{y})$ is a differential polynomial in the input/output variables $\textbf{u}, \bf{\dot{u}},\bf{\ddot{u}},\bf{\dddot{u}}$, $\ldots$, $\textbf{y}, \bf{\dot{y}},\bf{\ddot{y}},\bf{\dddot{y}},\ldots$, etc.  If the values of $\textbf{u}, \bf{\dot{u}},\bf{\ddot{u}},\bf{\dddot{u}}$,$\ldots$, $\textbf{y}, \bf{\dot{y}},\bf{\ddot{y}},\bf{\dddot{y}},\ldots$, etc, were known at a sufficient number of time instances $t_1,t_2,...,t_m$, then one could substitute in values of $\psi_{i}(\textbf{u},\textbf{y})$ and $\xi (\textbf{u},\textbf{y})$ at each of these time instances to obtain a linear system of equations in the variables $\tilde{c}_{i}(\textbf{p})$.  

First consider the case of a single input-output equation.  If there are $n$ unknown coefficients $\tilde{c}_{i}(\textbf{p})$, we obtain the system:
\begin{align*}
 \tilde{c}_{1}(\textbf{p})\psi_{1}(\textbf{u}(t_1),\textbf{y}(t_1))+...+&\tilde{c}_{n}(\textbf{p})\psi_{n}(\textbf{u}(t_1),\textbf{y}(t_1))= \xi (\textbf{u}(t_1),\textbf{y}(t_1)) \\
&...\\
 \tilde{c}_{1}(\textbf{p})\psi_{1}(\textbf{u}(t_m),\textbf{y}(t_m))+...+&\tilde{c}_{n}(\textbf{p})\psi_{n}(\textbf{u}(t_m),\textbf{y}(t_m))= \xi (\textbf{u}(t_m),\textbf{y}(t_m)) \\
\end{align*}

We write this linear system as $A \kappa = b$, where $A$ is an $m$ by $n$ matrix of the form:
$$
\begin{pmatrix}
\psi_{1}(\textbf{u}(t_1),\textbf{y}(t_1)) & ... & \psi_{n}(\textbf{u}(t_1),\textbf{y}(t_1)) \\
\vdots & \vdots & \vdots \\
\psi_{1}(\textbf{u}(t_m),\textbf{y}(t_m)) & ... & \psi_{n}(\textbf{u}(t_m),\textbf{y}(t_m))
\end{pmatrix}
$$
$\kappa$ is the vector of unknown coefficients $[\tilde{c}_{1}(\textbf{p}),...,\tilde{c}_{n}(\textbf{p})]^T$, and $b$ is of the form $[\xi (\textbf{u}(t_1),\textbf{y}(t_1)), ..., \xi (\textbf{u}(t_m),\textbf{y}(t_m))]^T$.  

For the case of multiple input-output equations, we get the following block diagonal system of equations $A \kappa = b$:
$$
\begin{pmatrix}
A_1 & 0 & 0 & \ldots & 0 \\
0 & A_2 & 0 & \ldots & 0 \\
\vdots & \vdots & \ddots & \vdots & \vdots \\
0 & 0 & 0 & \ldots & A_M
\end{pmatrix}
\begin{pmatrix}
\kappa_1 \\
\kappa_2 \\
\vdots \\
\kappa_M
\end{pmatrix}
=
\begin{pmatrix}
b_1 \\
b_2 \\
\vdots \\
b_M
\end{pmatrix}
$$
where $A$ is a $m=m_1+...+m_M$ by $n=n_1+...+n_M$ matrix.

For noise-free (perfect) data, this system $A \kappa = b$ should have a unique solution for $\kappa$ \cite{Ljung}.  In other words, the coefficients $\tilde{c}_{i}(\textbf{p})$ of the input-output equations can be uniquely determined from enough input/output data \cite{Ljung}. 

The main idea of this paper is the following.  Given a set of candidate models, we find their associated differential invariants and then substitute in values of $\textbf{u}, \bf{\dot{u}},\bf{\ddot{u}},\bf{\dddot{u}},\ldots, \textbf{y}, \bf{\dot{y}},\bf{\ddot{y}},\bf{\dddot{y}},\ldots$, etc, %known input/output variables and derivatives
at many time instances $t_1,...,t_m$, thus setting up the linear system $A \kappa = b$ for each model.  The solution to $A \kappa = b$ should be unique for the correct model, but there should be no solution for each of the incorrect models.  Thus under ideal circumstances, one should be able to select the correct model since the input/output data corresponding to that model should satisfy its differential invariant.  Likewise, one should be able to reject the incorrect models since the input/output data should not satisfy their differential invariants.

However, with imperfect data, there could be no solution to $A \kappa = b$ even for the correct model.  Thus, with imperfect data, one may be unable to select the correct model.  On the other hand, if there is no solution to $A \kappa = b$ for each of the candidate models, then the goal is to determine how ``badly'' each of the models fail and reject models accordingly.  We now describe criteria to reject models.

\section{Solvability of linear system $A \kappa = B$}

 Let $A \in \mathbb{R}^{m \times n}$ and consider the linear system
 \begin{align}
  A\kappa = B,
  \label{eqn:linear-system}
 \end{align}
 where $B \in \mathbb{R}^{m \times r}$.  Note, in our case, $r=1$, so $B$ is just the vector $b$.  Here, we study the solvability of \eqref{eqn:linear-system} under (a specific form of) perturbation of both $A$ and $B$. Let $\tilde{A}$ and $\tilde{B}$ denote the perturbed versions of $A$ and $B$, respectively, and assume that $\tilde{A} - A$ and $\tilde{B} - B$ depend only on $\tilde{A}$ and $\tilde{B}$, respectively. Our goal is to infer the {\em unsolvability} of the unperturbed system \eqref{eqn:linear-system} from observation of $\tilde{A}$ and $\tilde{B}$ only.
 
 We will describe how to detect the rank of an augmented matrix, but first introduce notation. The singular values of a matrix $A \in \mathbb{R}^{m \times n}$ will be denoted by
 \begin{align*}
  \sigma_{1} (A) \geq \cdots \geq \sigma_{\ell} (A) \geq \sigma_{\ell + 1} (A) = \cdots = \sigma_{n} (A) = 0, \quad \ell = \min (m, n).
 \end{align*}
 (Note that we have trivially extended the number of singular values of $A$ from $\ell$ to $n$.) The rank of $A$ is written $\rank (A)$. The range of $A$ is denoted $\range (A)$. Throughout, $\| \cdot \|$ refers to the Euclidean norm.

 The basic strategy will be to assume as a null hypothesis that \eqref{eqn:linear-system} has a solution, i.e., $B \in \range (A)$, and then to derive its consequences in terms of $\tilde{A}$ and $\tilde{B}$. If these consequences are not met, then we conclude by contradiction that \eqref{eqn:linear-system} is unsolvable. In other words, we will provide {\em sufficient but not necessary} conditions for \eqref{eqn:linear-system} to have no solution, i.e., we can only reject (but not confirm) the null hypothesis. We will refer to this procedure as {\em testing} the null hypothesis.

 \subsection{Preliminaries}
 We first collect some useful results. The first, Weyl's inequality, is quite standard.

 \begin{theorem}[Weyl's inequality]
  Let $A, B \in \mathbb{R}^{m \times n}$. Then
  \begin{align*}
   |\sigma_{k} (A) - \sigma_{k} (B)| \leq \| A - B \|, \quad k = 1, \dots, n.
  \end{align*}
 \end{theorem}
 Weyl's inequality can be used to test $\rank (A)$ using knowledge of only $\tilde{A}$.

 \begin{corollary}
  Let $A, \tilde{A} \in \mathbb{R}^{m \times n}$ and assume that $\rank (A) < k$. Then
  \begin{align}
   \sigma_{k} (\tilde{A}) \leq \| \tilde{A} - A \|.
   \label{eqn:weyl-rank}
  \end{align}
  \label{cor:weyl-rank}
 \end{corollary}

 Therefore, if \eqref{eqn:weyl-rank} is not satisfied, then $\rank (A) \geq k$.

 \subsection{Augmented matrix}
 \label{sec:augmented-matrix}
 Assume the null hypothesis. Then $B \in \range (A)$, so $\rank ([A, B]) = \rank (A) \leq \min (m, n)$. Therefore, $\sigma_{n + 1} ([A, B]) = 0$. But we do not have access to $[A, B]$ and so must consider instead the perturbed augmented matrix $[\tilde{A}, \tilde{B}]$.

 \begin{theorem}
  Under the null hypothesis,
  \begin{align}
   \sigma_{n + 1} ([\tilde{A}, \tilde{B}]) \leq \| [\tilde{A} - A, \tilde{B} - B] \| \leq \| \tilde{A} - A \| + \| \tilde{B} - B \|.
   \label{eqn:augmented-sigma}
  \end{align}
  \label{thm:augmented-matrix}
 \end{theorem}

 \begin{proof}
  Apply Corollary \ref{cor:weyl-rank}.
 \end{proof}

 In other words, if \eqref{eqn:augmented-sigma} does not hold, then \eqref{eqn:linear-system} has no solution.

 \begin{remark}
  This approach can fail to correctly reject the null hypothesis if $A$ is (numerically) low-rank. As an example, suppose that $\rank (A) < n$ and let $B \notin \range (A)$ consist of a single vector ($p = 1$). Then $\rank ([A, B]) \leq n$, so $\sigma_{n + 1} ([A, B]) = 0$ (or is small). Assuming that $\| \tilde{A} - A \|$ and $\| \tilde{B} - B \|$ are small, $\sigma_{n + 1} ([\tilde{A}, \tilde{B}])$ will hence also be small.
 \end{remark}

 \begin{remark}
  In principle, we should test directly the assertion that $\rank ([A, B]) = \rank (A)$. However, we can only establish lower bounds on the matrix rank (we can only tell if a singular value is ``too large''), so this is not feasible in practice. An alternative approach is to consider only {\em numerical} ranks obtained by thresholding. How to choose such a threshold, however, is not at all clear and can be a very delicate matter especially if the data have high dynamic range.
 \end{remark}

 \begin{remark}
  The theorem is uninformative if $m \leq n$ since then $\sigma_{n + 1} ([A, B]) = \sigma_{n + 1} (\tilde{A}, \tilde{B}) = 0$ trivially. However, this is not a significant disadvantage beyond that described above since if $A$ is full-rank, then it must be true that \eqref{eqn:linear-system} is solvable.
 \end{remark}

\subsection{Example: Perfect data}
As a proof of principle, we first apply Theorem \ref{thm:augmented-matrix} to a simple linear model.  We start by taking perfect input and output data and then add a specific amount of noise to the output data and attempt to reject the incorrect model.  In the subsequent sections, we will see how to interpret Theorem \ref{thm:augmented-matrix} statistically  under a particular ``noise'' model for the perturbations.

Here, we take data from a linear 3-compartment model, add noise, and try to reject the general form of the linear 2-compartment model with the same input/output compartments.

\begin{ex} \label{ex:mainex} 
Let our model be a 3-compartment model of the following form:
$$
\begin{pmatrix} 
\dot{x}_1 \\
\dot{x}_2 \\
\dot{x}_3 \end{pmatrix} = {\begin{pmatrix} 
-2 & 1 & 0 \\
1 & -3 & 1 \\
0 & 1 & -2  
\end{pmatrix}} {\begin{pmatrix}
x_1 \\
x_2 \\
x_3 \end{pmatrix} } + {\begin{pmatrix}
2e^{-3t} + 12e^{-5t} \\
0 \\
0 \end{pmatrix}}, \quad y=x_1$$
$$x_1(0)=1,x_2(0)=7,x_3(0)=9$$
Here we have an input to the first compartment of the form $u_1=2e^{-3t} + 12e^{-5t}$ and the first compartment is measured, so that $y=x_1$ represents the output.  The solution to this system of ODEs can be easily found of the form:
$$
\begin{pmatrix} 
x_1 \\
x_2 \\
x_3 \end{pmatrix} = {7\begin{pmatrix} 
1 \\
1 \\
1  
\end{pmatrix} e^{-t}} + {\begin{pmatrix}
-1 \\
0 \\
1 \end{pmatrix} e^{-2t}} + {\begin{pmatrix}
1 \\
-2 \\
1 \end{pmatrix}e^{-4t}} + {\begin{pmatrix}
-1 \\
-1 \\
1 \end{pmatrix}e^{-3t}}+ {\begin{pmatrix}
-5 \\
3 \\
-1 \end{pmatrix}e^{-5t}}
$$

\end{ex}
so that $y=7e^{-t}-e^{-2t}+e^{-4t}-e^{-3t}-5e^{-5t}$.

The input-output equation for a $3$ compartment model with a single input/output to the first compartment has the form:
$$\dddot{y}+c_1\ddot{y}+c_2\dot{y}+c_3y=\ddot{u}_1+c_4\dot{u}_1+c_5u_1$$
where $c_1,c_2,c_3$ are the coefficients of the characteristic polynomial of the matrix $A$ and $c_4,c_5$ are the coefficients of the characteristic polynomial of the matrix $A_1$ which has the first row and first column of $A$ removed.

We now substitute values of $u_1,\dot{u}_1,\ddot{u}_1,y,\dot{y},\ddot{y},\dddot{y}$ at time instances $t=0,0.2,0.4,0.6,0.8,1$ into our input-output equation and solved the resulting linear system of equations for $c_1,c_2,c_3,c_4,c_5$.  We get that $c_1=7,c_2=14,c_3=8,c_4=5,c_5=5$, which agrees with the coefficients of the characteristic polynomials of $A$ and $A_1$.  %The singular values for the matrix $A$ with the substituted values of $u_1,\dot{u}_1,y,\dot{y},\ddot{y}$ at time instances $t=0,0.2,0.4,0.6,0.8,1$ are:
%$$144.183, 8.89432, 0.352884, 0.0150481, 0.000111972$$

%The singular values of the matrix $(A|b)$ with the substituted values of $u_1,\dot{u}_1,\ddot{u}_1,y,\dot{y},\ddot{y},\dddot{y}$ at time instances $t=0,0.2,0.4,0.6,0.8,1$ are:
%$$323.277, 8.89434, 1.35614, 0.0220602, 0.00016819$$

%We now add noise to our matrix A in the following way.  To each entry $\ddot{y}$, $\dot{y}$, and $y$, we add $\epsilon k_{ij}$ where $k_{ij}$ is a random real number between $0$ and $1$, and $\epsilon$ equals $0.1,0.01,0.001$, respectively.  Then the noisy matrix $\tilde{A}$ has the following singular values:
%$$144.112, 8.90338, 0.32493, 0.0134679, 0.00170915$$

%We now add noise to our vector $b$ in the following way.  To each entry $\ddot{u}_1-\dddot{y}$, we add $\epsilon k_{ij}$ where $k_{ij}$ is a random real number between $0$ and $1$, and $\epsilon$ equals $1$.  Then the noisy matrix $\tilde{(A|b)}$ has the following singular values:
%$$322.629, 8.90906, 1.12804, 0.0582977, 0.0127956, 0.00162734$$

%Since the Frobenius norm of $\tilde{(A|b)}-(A|b)$ is $1.40311$, which is greater than the smallest singular value $0.00162734$, we are unable to reject this model.

We now attempt to reject the 2-compartment model using 3-compartment model data.  We find the input-output equations for a $2$ compartment model with a single input/output to the first compartment, which has the form:
$$\ddot{y}+c_2\dot{y}+c_3y=\dot{u}_1+c_5u_1$$
where again $c_2,c_3$ are the coefficients of the characteristic polynomial of the matrix $A$ and $c_5$ is the coefficient of the characteristic polynomial of the matrix $A_1$ which has the first row and first column of $A$ removed.  

We substitute values of $u_1,\dot{u}_1,y,\dot{y},\ddot{y}$ at time instances $t=0,0.2,0.4,0.6,0.8$ into our input-output equation and attempt to solve the resulting linear system of equations for $c_2,c_3,c_5$.  

The singular values for the matrix $A$ with the substituted values of $u_1,y,\dot{y}$ at time instances $t=0,0.2,0.4,0.6,0.8$ are:
$$24.7762, 7.10169, 0.0559192$$
The singular values of the matrix $(A|b)$ with the substituted values of $u_1,\dot{u}_1,y,\dot{y},\ddot{y}$ at time instances $t=0,0.2,0.4,0.6,0.8$ are:
$$57.1337, 7.13319, 0.279458, 0.00364017$$
We add noise to our matrix A in the following way.  To each entry $\dot{y}$, and $y$, we add $\epsilon k_{ij}$ where $k_{ij}$ is a random real number between $0$ and $1$, and $\epsilon$ equals $0.001$.  Then the noisy matrix $\tilde{A}$ has the following singular values:
$$24.7768, 7.10172, 0.0557071$$
We now add noise to our vector $b$ in the following way.  To each entry $\dot{u}_1-\ddot{y}$, we add $\epsilon k_{ij}$ where $k_{ij}$ is a random real number between $0$ and $1$, and $\epsilon$ equals $0.001$.  Then the noisy matrix $\tilde{(A|b)}$ has the following singular values:
$$57.1345, 7.13319, 0.279467, 0.00322281$$
We find the matrix $\tilde{(A|b)}-(A|b)$ and compare the norm of this matrix to the smallest singular value of $\tilde{(A|b)}$.  Since the Frobenius norm of $\tilde{(A|b)}-(A|b)$ is %$0.129907$
$0.0020603$, which is \textit{less than} the smallest singular value % $0.00109293$
$0.00322281$, we can reject this model.  Thus, using noisy 3-compartment model data, we are able to reject the 2-compartment model.

 \section{Statistical inference}
 \label{sec:stats}

 We now consider the statistical inference of the solvability of \eqref{eqn:linear-system}. First, we need a noise model.

 \subsection{Noise model}
 If the perturbations $\| \tilde{A} - A \|$ and $\| \tilde{B} - B \|$ are bounded, e.g., $\| \tilde{A} - A \| \leq \epsilon \| \tilde{A} \|$ and $\| \tilde{B} - B \| \leq \epsilon \| \tilde{B} \|$ for some $\epsilon > 0$ (representing a relative accuracy of $\epsilon$ in the ``measurements'' $\tilde{A}$ and $\tilde{B}$), then Theorem \ref{thm:augmented-matrix} can be used at once. However, it is customary to model such perturbations as normal random variables, which are not bounded. Here, we will assume a noise model of the form
 \begin{align*}
  \tilde{A} - A = \epsilon C_{\tilde{A}} \circ Z, \quad \tilde{B} - B = \epsilon C_{\tilde{B}} \circ Z,
 \end{align*}
 where $C_{\tilde{A}}$ is a (computable) matrix that depends on $\tilde{A}$ and similarly with $C_{\tilde{B}}$, $A \circ B$ denotes the Hadamard (entrywise) matrix product $(A \circ B)_{ij} = A_{ij} B_{ij}$, and $Z$ is a matrix-valued random variable whose entries $Z_{ij} \sim \normdist (0, 1)$ are independent standard normals.

 In our application of interest, the entries of $C_{\tilde{A}}$ depend on those of $\tilde{A}$ as follows. Let $A_{ij} = \phi_{ij} (x)$ for some input vector $x$ but suppose that we can only observe the ``noisy'' vector $\tilde{x} = (1 + \epsilon Z) \circ x$. Then the corresponding perturbed matrix entries are
 \begin{align*}
  \tilde{A}_{ij} = \phi_{ij} (\tilde{x}) = \phi_{ij} (x) + \epsilon \sum_{k} (\nabla \phi_{ij} (x))_{k} \, x_{k} \, Z_{k} + O(\epsilon^{2}), \quad Z_{k} \sim \normdist (0, 1).
 \end{align*}
 By the additivity formula
 \begin{align}
  \sum_{k} a_{k} Z_{k} = \sqrt{\sum_{k} a_{k}^{2}} \, Z = \| a \| Z
  \label{eqn:gaussian-sum}
 \end{align}
 for standard Gaussians\footnote{There is an error in \cite{harrington-pnas-2012}, in which we incorrectly used that $a Z_{1} + b Z_{2} = (a + b)Z$. However, the statistical conclusion is still valid since $(a + b)Z$ ``dominates'' $\sqrt{a^{2} + b^{2}} \, Z$ in the sense that the former has variance $a^{2} + 2ab + b^{2}$, while the latter has variance only $a^{2} + b^{2}$. In other words, we were wrong but in the conservative direction. This was taken into account in \cite{maclean2015}.}
 \begin{align*}
  \sum_{k} (\nabla \phi_{ij} (x))_{k} \, x_{k} \, Z_{k} = \sum_{k} (\nabla \phi_{ij} (x) \circ x)_{k} \, Z = \| \nabla \phi_{ij} (x) \circ x \| Z.
 \end{align*}
 Therefore,
 \begin{align*}
  \tilde{A}_{ij} = A_{ij} + \epsilon \| \nabla \phi_{ij} (x) \circ x \| Z + O(\epsilon^{2}) = A_{ij} + \epsilon \| \nabla \phi_{ij} (\tilde{x}) \circ \tilde{x} \| Z + O(\epsilon^{2}),
 \end{align*}
 so, to first order in $\epsilon$,
 \begin{align*}
  (C_{\tilde{A}})_{ij} = \| \nabla \phi_{ij} (\tilde{x}) \circ \tilde{x} \|.
 \end{align*}
 An analogous derivation holds for $C_{\tilde{B}}$.

 Each of the bounds in the theorems above are linear in $\| \tilde{A} - A \|$ and $\| \tilde{B} - B \|$ (for Theorem \ref{thm:augmented-matrix}, the bound is simply the sum of these two) and so may be written as $\| C_{\tilde{A}} \circ Z \| + \| C_{\tilde{B}} \circ Z \|$ by absorbing constants.

 The basic strategy is now as follows. Let $\tau$ be a test statistic, i.e., $\sigma_{n + 1} ([\tilde{A}, \tilde{B}])$ in \S \ref{sec:augmented-matrix}. Then since
 \begin{align*}
  \tau_{\omega} \leq (\| C_{\tilde{A}} \circ Z \| + \| C_{\tilde{B}} \circ Z \|)_{\omega},
 \end{align*}
 where we have made explicit the dependence of both sides on the same underlying random mechanism $\omega$, the (cumulative) distribution function of $\tau$ must dominate that of $\| C_{\tilde{A}} \circ Z \| + \| C_{\tilde{B}} \circ Z \|$, i.e.,
 \begin{align*}
  \Pr (\tau \leq x) \geq \Pr (\| C_{\tilde{A}} \circ Z \| + \| C_{\tilde{B}} \circ Z \| \leq x).
 \end{align*}
 Thus,
 \begin{subequations}
  \begin{align}
   \Pr (\tau \geq x) &\leq \Pr (\| C_{\tilde{A}} \circ Z \| + \| C_{\tilde{B}} \circ Z \| \geq x)\\
   &= \int_{0}^{\infty} \Pr (\| C_{\tilde{A}} \circ Z \| = t) \Pr (\| C_{\tilde{B}} \circ Z \| \geq x - t) \, dt\\
   &= \int_{0}^{x} \Pr (\| C_{\tilde{A}} \circ Z \| = t) \Pr (\| C_{\tilde{B}} \circ Z \| \geq x - t) \, dt + \int_{x}^{\infty} \Pr (\| C_{\tilde{A}} \circ Z \| = t) \, dt\\
   &\leq \int_{0}^{x} \Pr (\| C_{\tilde{A}} \circ Z \| \geq t) \Pr (\| C_{\tilde{B}} \circ Z \| \geq x - t) \, dt + \Pr (\| C_{\tilde{A}} \circ Z \| \geq x).\label{eqn:prob-tau-rhs}
  \end{align}
  \label{eqn:prob-tau}
 \end{subequations}
 Note that if, e.g., $\tilde{B} = B$ (i.e., if $B$ were known exactly), then \eqref{eqn:prob-tau-rhs} simplifies to just $\Pr (\| C_{\tilde{A}} \circ Z \| \geq x)$.

 Using \eqref{eqn:prob-tau}, we can associate a $p$-value to any given realization of $\tau$ by referencing upper tail bounds for quantities of the form $\| C \circ Z \|$. Recall that $\tau = 0$ under the null hypothesis. In a classical statistical hypothesis testing framework, we may therefore reject the null hypothesis if \eqref{eqn:prob-tau-rhs} is at most $\alpha$, where $\alpha$ is the desired significance level (e.g., $\alpha = 0.05$).

 \subsection{Hadamard tail bounds}
 We now turn to bounding $\Pr (\| C \circ Z \| \geq x)$, where we will assume that $C, Z \in \mathbb{R}^{m \times n}$. This can be done in several ways.

 One easy way is to recognize that
 \begin{align}
  \| C \circ Z \| \leq \| C \circ Z \|_{F} \leq \| C \|_{F} \| Z \|_{F},
  \label{eqn:noise-frobenius}
 \end{align}
 where $\| \cdot \|_{F}$ is the Frobenius norm, so
 \begin{align*}
  \Pr (\| C \circ Z \| \geq x) \leq \Pr (\| C \circ Z \|_{F} \geq x) \leq \Pr \left( \| Z \|_{F} \geq \frac{x}{\| C \|_{F}} \right).
 \end{align*}
 But $\| Z \|_{F} \sim \chi_{mn}$ has a chi distribution\footnote{Note that this is {\em not} the chi-squared distribution (though $\| Z \|_{F}^{2} \sim \chi^{2}_{mn}$).} with $mn$ degrees of freedom. Therefore,
 \begin{align*}
  \Pr \left( \| Z \|_{F} \geq \frac{x}{\| C \|_{F}} \right) = \Pr \left( X \geq \frac{x}{\| C \|_{F}} \right), \quad X \sim \chi_{mn}.
 \end{align*}
 However, each inequality in \eqref{eqn:noise-frobenius} can be quite loose: The first is loose in the sense that
 \begin{align*}
  \| A \|^{2} = \sigma_{1}^{2} (A), \quad \| A \|_{F}^{2} = \sum_{k} \sigma_{k}^{2} (A);
 \end{align*}
 while the second in that
 \begin{align*}
  \left\|
  \begin{bmatrix}
   a_{1}\\
   a_{2}
  \end{bmatrix} \circ
  \begin{bmatrix}
   b_{1}\\
   b_{2}
  \end{bmatrix} \right\|_{F}^{2} = \left\|
  \begin{bmatrix}
   a_{1} b_{1}\\
   a_{2} b_{2}
  \end{bmatrix} \right\|_{F}^{2} = (a_{1} b_{1})^{2} + (a_{2} b_{2})^{2},
 \end{align*}
 but
 \begin{align*}
  \left\|
  \begin{bmatrix}
   a_{1}\\
   a_{2}
  \end{bmatrix} \right\|_{F}^{2} \left\|
  \begin{bmatrix}
   b_{1}\\
   b_{2}
  \end{bmatrix} \right\|_{F}^{2} = (a_{1}^{2} + a_{2}^{2})(b_{1}^{2} + b_{2}^{2}) = (a_{1} b_{1})^{2} + (a_{1} b_{2})^{2} + (a_{2} b_{1})^{2} + (a_{2} b_{2})^{2}.
 \end{align*}

 A slightly better approach is to use the inequality \cite{zhan:1997:siam-j-matrix-anal-appl}
 \begin{align*}
  \| C \circ Z \| \leq \| \min (\max_{i} \| C_{i,:} \|, \max_{j} \| C_{:,j} \|) \, \| Z \|,
 \end{align*}
 where $C_{i,:}$ and $C_{:,j}$ denote the $i$th row and $j$th column, respectively, of $C$. The $\| Z \|$ term can then be handled using a chi distribution via $\| Z \| \leq \| Z \|_{F}$ as above or directly using a concentration bound (see below). Variations on this undoubtedly exist.

 Here, we will appeal to a result by Tropp \cite{tropp:2012:found-comput-math}. The following is from \S 4.3 in \cite{tropp:2012:found-comput-math}.

 \begin{theorem}
  Let $C, Z \in \mathbb{R}^{m \times n}$, where each $Z_{ij} \sim \normdist (0, 1)$. Then for any $x \geq 0$,
  \begin{align*}
   \Pr (\| C \circ Z \| \geq x) \leq (m + n) \exp \left( -\frac{x^{2}}{2 \sigma^{2}} \right), \quad \sigma^{2} = \max (\max_{i} \| C_{i,:} \|^{2}, \max_{j} \| C_{:,j} \|^{2}).
  \end{align*}
  \label{thm:hadamard-gaussian}
 \end{theorem}

 \subsection{Test statistic tail bounds}
 The bound \eqref{eqn:prob-tau-rhs} for $\Pr (\tau \geq x)$ can then be computed as follows. Let
 \begin{align*}
  P_{1} (x) = \int_{0}^{x} \Pr (\| C_{\tilde{A}} \circ Z \| \geq t) \Pr (\| C_{\tilde{B}} \circ Z \| \geq x - t) \, dt, \quad P_{2} (x) = \Pr (\| C_{\tilde{A}} \circ Z \| \geq x)
 \end{align*}
 so that $\Pr (\tau \geq x) \leq P_{1} (x) + P_{2} (x)$. Then by Theorem \ref{thm:hadamard-gaussian},
 \begin{align*}
  P_{1} (x) \leq (m + n)^{2} \int_{0}^{x} \exp \left[ -\frac{1}{2} \left( \frac{t^{2}}{\sigma_{A}^{2}} + \frac{(x - t)^{2}}{\sigma_{B}^{2}} \right) \right] \, dt,
 \end{align*}
 where $\sigma_{A}^{2}$ and $\sigma_{B}^{2}$ are the ``variance'' parameters in the theorem for $C_{\tilde{A}}$ and $C_{\tilde{B}}$, respectively. The term in parentheses simplifies to
 \begin{align*}
  \frac{t^{2}}{\sigma_{A}^{2}} + \frac{(x - t)^{2}}{\sigma_{B}^{2}} &= \frac{1}{\sigma_{A}^{2} \sigma_{B}^{2}} \left[ (\sigma_{A}^{2} + \sigma_{B}^{2}) t^{2} - 2 \sigma_{A}^{2} tx + \sigma_{A}^{2} x^{2} \right]\\
  &= \frac{1}{\sigma_{A}^{2} \sigma_{B}^{2}} \left[ (\sigma_{A}^{2} + \sigma_{B}^{2}) \left( t - \frac{\sigma_{A}^{2}}{\sigma_{A}^{2} + \sigma_{B}^{2}} x \right)^{2} + \sigma_{A}^{2} \left( 1 - \frac{\sigma_{A}^{2}}{\sigma_{A}^{2} + \sigma_{B}^{2}} \right) x^{2} \right]\\
  &= \frac{1}{\sigma_{A}^{2} \sigma_{B}^{2}} \left[ (\sigma_{A}^{2} + \sigma_{B}^{2}) \left( t - \frac{\sigma_{A}^{2}}{\sigma_{A}^{2} + \sigma_{B}^{2}} x \right)^{2} + \frac{\sigma_{A}^{2} \sigma_{B}^{2}}{\sigma_{A}^{2} + \sigma_{B}^{2}} x^{2} \right]\\
  &= \frac{\sigma_{A}^{2} + \sigma_{B}^{2}}{\sigma_{A}^{2} \sigma_{B}^{2}} \left( t - \frac{\sigma_{A}^{2}}{\sigma_{A}^{2} + \sigma_{B}^{2}} x \right)^{2} + \frac{x^{2}}{\sigma_{A}^{2} + \sigma_{B}^{2}}
 \end{align*}
 on completing the square. Therefore,
 \begin{align*}
  P_{1} (x) \leq (m + n)^{2} \exp \left[ -\frac{1}{2} \left( \frac{x^{2}}{\sigma_{A}^{2} + \sigma_{B}^{2}} \right) \right] \int_{0}^{x} \exp \left[ -\frac{1}{2} \left( \frac{\sigma_{A}^{2} + \sigma_{B}^{2}}{\sigma_{A}^{2} \sigma_{B}^{2}} \right) \left( t - \frac{\sigma_{A}^{2}}{\sigma_{A}^{2} + \sigma_{B}^{2}} x \right)^{2} \right] dt.
 \end{align*}
 Now set
 \begin{align*}
  \sigma^{2} = \frac{\sigma_{A}^{2} \sigma_{B}^{2}}{\sigma_{A}^{2} + \sigma_{B}^{2}}, \quad \alpha = \frac{\sigma_{A}^{2}}{\sigma_{A}^{2} + \sigma_{B}^{2}}
 \end{align*}
 so that the integral becomes
 \begin{align*}
  \int_{0}^{x} \exp \left[ -\frac{1}{2} \left( \frac{\sigma_{A}^{2} + \sigma_{B}^{2}}{\sigma_{A}^{2} \sigma_{B}^{2}} \right) \left( t - \frac{\sigma_{A}^{2}}{\sigma_{A}^{2} + \sigma_{B}^{2}} x \right)^{2} \right] dt = \int_{0}^{x} \exp \left[ -\frac{(t - \alpha x)^{2}}{2 \sigma^{2}} \right] dt.
 \end{align*}
 The variable substitution $u = (t - \alpha x)/\sigma$ then gives
 \begin{align*}
  \int_{0}^{x} \exp \left[ -\frac{(t - \alpha x)^{2}}{2 \sigma^{2}} \right] dt = \sigma \int_{-\alpha x / \sigma}^{(1 - \alpha) x / \sigma} e^{-u^{2}/2} \, du = \sqrt{2 \pi} \sigma \left[ \Phi \left( \frac{(1 - \alpha) x}{\sigma} \right) - \Phi \left( -\frac{\alpha x}{\sigma} \right) \right],
 \end{align*}
 where
 \begin{align*}
  \Phi (x) = \frac{1}{\sqrt{2 \pi}} \int_{-\infty}^{x} e^{-t^{2}/2} \, dt
 \end{align*}
 is the standard normal distribution function. Thus,
 \begin{align}
  P_{1} (x) \leq \sqrt{2 \pi} \sigma (m + n)^{2} \left[ \Phi \left( \frac{(1 - \alpha) x}{\sigma} \right) - \Phi \left( -\frac{\alpha x}{\sigma} \right) \right] \exp \left[ -\frac{1}{2} \left( \frac{x^{2}}{\sigma_{A}^{2} + \sigma_{B}^{2}} \right) \right].
  \label{eqn:p1}
 \end{align}
 A similar (but much simpler) analysis yields
 \begin{align}
  P_{2} (x) \leq (m + n) \exp \left( -\frac{x^{2}}{2 \sigma_{A}^{2}} \right).
  \label{eqn:p2}
 \end{align}

\section{Gaussian Processes to estimate derivatives}

We next present a method for estimating higher order derivatives and the estimation error using Gaussian Process regression and then apply the differential invariant method to both linear and nonlinear models in the subsequent sections.

A Gaussian process (GP) is a stochastic process $X(t) \sim \normdist (\mu(t), \Sigma(t,t'))$, where $\mu(t)$ is a mean function and $\Sigma(t,t')$ a covariance function. GPs are often used for regression/prediction as follows.

 Suppose that there is an underlying deterministic function $x(t)$ that we can only observe with some measurement noise as $\hat{x} (t) = x(t) + \epsilon(t)$, where $\epsilon(t) \sim \normdist(0, \sigma^{2} (t) \delta (t, t'))$ for
 \begin{align*}
  \delta(t, t') =
  \begin{cases}
   1 & \text{if $t = t'$}\\
   0 & \text{if $t \neq t'$}
  \end{cases}
 \end{align*}
 the Dirac delta. We consider the problem of finding $x(t)$ in a Bayesian setting by assuming it to be a GP with prior mean and covariance functions $\mu_{\prior}$ and $\Sigma_{\prior}$, respectively. Then the joint distribution of $\hat{x} (\boldsymbol{t}) = [\hat{x} (t_{1}), \dots, \hat{x} (t_{p})]^{\trans}$ at the observation points $\boldsymbol{t} = [t_{1}, \dots, t_{p}]^{\trans}$ and $x(\boldsymbol{s}) = [x(s_{1}), \dots, x(s_{q})]^{\trans}$ at the prediction points $\boldsymbol{s} = [s_{1}, \dots, s_{q}]^{\trans}$ is
 \begin{align}
  \begin{bmatrix}
   \hat{x} (\boldsymbol{t})\\
   x(\boldsymbol{s})
  \end{bmatrix} \sim \normdist \left(
  \begin{bmatrix}
   \mu_{\prior} (\boldsymbol{t})\\
   \mu_{\prior} (\boldsymbol{s})
  \end{bmatrix},
  \begin{bmatrix}
   \Sigma_{\prior} (\boldsymbol{t}, \boldsymbol{t}) + \sigma^{2} (\boldsymbol{t}) I & \Sigma_{prior}^{\trans} (\boldsymbol{s}, \boldsymbol{t})\\
   \Sigma_{\prior} (\boldsymbol{s}, \boldsymbol{t}) & \Sigma_{\prior} (\boldsymbol{s}, \boldsymbol{s})
  \end{bmatrix}
  \right).
  \label{eqn:joint-dist}
 \end{align}
 The conditional distribution of $x(\boldsymbol{s})$ given $x(\boldsymbol{t}) = \hat{x} (\boldsymbol{t})$ is also Gaussian:
 \begin{align}
  x(\boldsymbol{s}) \: | \: (x(\boldsymbol{t}) = \hat{x} (\boldsymbol{t})) \sim \normdist (\mu_{\post}, \Sigma_{\post}),
  \label{eqn:gpr}
 \end{align}
 where
 \begin{align*}
  \mu_{\post} &= \mu_{\prior} (\boldsymbol{s}) + \Sigma_{\prior} (\boldsymbol{s}, \boldsymbol{t}) (\Sigma_{\prior} (\boldsymbol{t}, \boldsymbol{t}) + \sigma^{2} (\boldsymbol{t}) I)^{-1} (\hat{x} (\boldsymbol{t}) - \mu_{\prior} (\boldsymbol{t}))\\
  \Sigma_{\post} &= \Sigma_{\prior} (\boldsymbol{s}, \boldsymbol{s}) - \Sigma_{\prior} (\boldsymbol{s}, \boldsymbol{t}) (\Sigma_{\prior} (\boldsymbol{t}, \boldsymbol{t}) + \sigma^{2} (\boldsymbol{t}) I)^{-1}) \Sigma_{\prior}^{\trans} (\boldsymbol{s}, \boldsymbol{t})
 \end{align*}
 are the posterior mean and covariance, respectively. This allows us to infer $x(\boldsymbol{s})$ on the basis of observing $\hat{x} (\boldsymbol{t})$. The diagonal entries of $\Sigma_{\post}$ are the posterior variances and quantify the uncertainty associated with this inference procedure.

 \subsection{Estimating derivatives}
 Equation \eqref{eqn:gpr} provides an estimate for the function values $x(\boldsymbol{s})$. What if we want to estimate its derivatives? Let $\cov (x(t), x(t')) = k(t, t')$ for some covariance function $k$. Then $\cov (x^{(m)} (t), x^{(n)} (t')) = \partial_{t}^{m} \partial_{t'}^{n} k(t, t')$ by linearity of differentiation. Thus,
 \begin{align*}
  \begin{pmat}[{.}]
   \hat{x} (\boldsymbol{t}) \cr\-
   x(\boldsymbol{s}) \cr
   x'(\boldsymbol{s}) \cr
   \vdots \cr
   x^{(n)} (\boldsymbol{s}) \cr
  \end{pmat} \sim \normdist \left(
  \begin{pmat}[{.}]
   \mu_{\prior} (\boldsymbol{t}) \cr\-
   \mu_{\prior} (\boldsymbol{s}) \cr
   \mu_{\prior}^{(1)} (\boldsymbol{s}) \cr
   \vdots \cr
   \mu_{\prior}^{(n)} (\boldsymbol{s}) \cr
  \end{pmat},
  \begin{pmat}[{|...}]
   \Sigma_{\prior} (\boldsymbol{t}, \boldsymbol{t}) + \sigma^{2} (\boldsymbol{t}) I & \Sigma_{\prior}^{\trans} (\boldsymbol{s}, \boldsymbol{t}) & \Sigma_{\prior}^{(1,0),\trans} (\boldsymbol{s}, \boldsymbol{t}) & \cdots & \Sigma_{\prior}^{(n,0), \trans} (\boldsymbol{s}, \boldsymbol{t}) \cr\-
   \Sigma_{\prior} (\boldsymbol{s}, \boldsymbol{t}) & \Sigma_{\prior} (\boldsymbol{s}, \boldsymbol{s}) & \Sigma_{\prior}^{(1,0), \trans} (\boldsymbol{s}, \boldsymbol{s}) & \cdots & \Sigma_{\prior}^{(n,0), \trans} (\boldsymbol{s}, \boldsymbol{s}) \cr
   \Sigma_{\prior}^{(1,0)} (\boldsymbol{s}, \boldsymbol{t}) & \Sigma_{\prior}^{(1,0)} (\boldsymbol{s}, \boldsymbol{s}) & \Sigma_{\prior}^{(1,1)} (\boldsymbol{s}, \boldsymbol{s}) & \cdots & \Sigma_{\prior}^{(n,1), \trans} (\boldsymbol{s}, \boldsymbol{s}) \cr
   \vdots & \vdots & \vdots & \ddots & \vdots \cr
   \Sigma_{\prior}^{(n,0)} (\boldsymbol{s}, \boldsymbol{t}) & \Sigma_{\prior}^{(n,0)} (\boldsymbol{s}, \boldsymbol{s}) & \Sigma_{\prior}^{(n,1)} (\boldsymbol{s}, \boldsymbol{s}) & \cdots & \Sigma_{(n,n)} (\boldsymbol{s}, \boldsymbol{s}) \cr
  \end{pmat} \right),
 \end{align*}
 where $\mu_{\prior}^{(i)} (t)$ is the prior mean for $x^{(i)} (t)$ and $\Sigma_{\prior}^{(i,j)} (t, t') = \partial_{t}^{i} \partial_{t'}^{j} \Sigma_{\prior} (t, t')$. This joint distribution is exactly of the form \eqref{eqn:joint-dist}. An analogous application of \eqref{eqn:gpr} then yields the posterior estimate of $x^{(i)} (\boldsymbol{s}) \: | \: (x(\boldsymbol{t}) = \hat{x} (\boldsymbol{t}))$ for all $i = 0, 1, \dots, n$.

 Alternatively, if we are interested only in the posterior variances of each $x^{(i)} (\boldsymbol{s})$, then it suffices to consider each $2 \times 2$ block independently:
 \begin{align*}
  \begin{bmatrix}
   \hat{x} (\boldsymbol{t})\\
   x^{(i)} (\boldsymbol{s})
  \end{bmatrix} \sim \normdist \left(
  \begin{bmatrix}
   \mu_{\prior} (\boldsymbol{t})\\
   \mu_{\prior}^{(i)} (\boldsymbol{s})
  \end{bmatrix},
  \begin{bmatrix}
   \Sigma_{\prior} (\boldsymbol{t}, \boldsymbol{t}) + \sigma^{2} (\boldsymbol{t}) I & \Sigma_{\prior}^{(i,0), \trans} (\boldsymbol{s}, \boldsymbol{t})\\
   \Sigma_{\prior}^{(i,0)} (\boldsymbol{s}, \boldsymbol{t}) & \Sigma_{\prior}^{(i,i)} (\boldsymbol{s}, \boldsymbol{s})
  \end{bmatrix} \right).
 \end{align*}
 The cost of computing $(\Sigma_{\prior} (\boldsymbol{t}, \boldsymbol{t}) + \sigma^{2} (\boldsymbol{t}) I)^{-1}$ can clearly be amortized over all $i$.

 \subsection{Formulae for squared exponential covariance functions}
 We now consider the specific case of the squared exponential (SE) covariance function
 \begin{align*}
  k(t, t') = \theta^{2} \exp \left[ -\frac{(t - t')^{2}}{2 \ell^{2}} \right],
 \end{align*}
 where $\theta^{2}$ is the signal variance and $\ell$ is a length scale. The SE function is one of the most widely used covariance functions in practice. Its derivatives can be expressed in terms of the (probabilists') Hermite polynomials
 \begin{align*}
  H_{n} (x) = (-1)^{n} e^{x^{2} / 2} \frac{d^{n}}{d x^{n}} e^{-x^{2} / 2}
 \end{align*}
 (these are also sometimes denoted $He_{n} (x)$). The first few Hermite polynomials are $H_{0} (x) = 1$, $H_{1} (x) = x$, and $H_{2} (x) = x^{2} - 1$.

 We need to compute the derivatives $\partial_{t}^{m} \partial_{t'}^{n} k(t, t')$. Let $v = (t - t') / \ell$ so that $k(t, t') = k(v) = \theta^{2} e^{-v^{2} / 2}$. Then $\partial_{t}^{m} f(v) = (1 / \ell)^{m} f^{(m)} (v)$ and $\partial_{t'}^{n} f(v) = (-1 / \ell)^{n} f^{(n)} (v)$. Therefore,
 \begin{align*}
  \frac{\partial^{m}}{\partial t^{m}} \frac{\partial^{n}}{\partial {t'}^{n}} k(t, t') = \frac{(-1)^{n}}{\ell^{m + n}} k^{(m + n)} (v) = \frac{(-1)^{m}}{\ell^{m + n}} H_{m + n} (v) k(v) = \frac{(-1)^{m}}{\ell^{m + n}} H_{m + n} \left( \frac{t - t'}{\ell} \right) k(t, t').
 \end{align*}

% \textcolor{blue}{Estimation of $\epsilon$? Description of hyper-parameters?}

The GP regression requires us to have the values of the hyperparameters $\sigma^2$, $\theta^2$, and $\ell$. In practice, however, these are hardly ever known. In the examples below, we deal with this by estimating the hyperparameters from the data by maximizing the likelihood. We do this by using a nonlinear conjugate gradient algorithm, which can be quite sensitive to the initial starting point, so we initialize multiple runs over a small grid in hyperparameter space and return the best estimate found. This increases the quality of the estimated hyperparameters but can still sometimes fail.

\section{Results}
 
We showcase our method on competing models: linear compartment models (2 and 3 species), Lotka-Volterra models (2 and 3 species) and Lorenz.  As the linear compartment differential invariants were presented in an earlier section, we compute the differential invariants of the Lotka-Volterra and Lorenz using {\tt RosenfeldGroebner}. We simulate each of these models to generate time course data, add varying levels of noise, and estimate the necessary higher order derivatives using GP regression. As described in the earlier section, we require the estimation of the higher order derivatives to satisfy a negative log likelihood value, otherwise the GP fit is not `good'. In some cases, this can be remedied by increase the number of data points. Using the estimated GP regression data, we test each of the models using the differential invariant method on other models.

\begin{ex} \label{ex:LV2} 
%\subsubsection{Lotka-Volterra two species system}
The two species Lotka-Volterra model is:
\begin{align*}
\dot{x}_1 &= p_1 x_1 - p_2 x_1x_2,\\
\dot{x}_2 &= -p_3 x_2+ p_4 x_1x_2,
\end{align*}
where $x_1$ and $x_2$ are variables, and $p_1,p_2,p_3,p_4$ are parameters. We assume only $x_1$ is  observable and perform differential elimination and obtain our differential invariant in terms of only $y=x_1(t)$:
$$p_4\dot{y}y^2 - p_3\dot{y}y - p_1 p_4 y^3 + p_1 p_3 y^2=\ddot{y}y - \dot{y}^2.$$

\end{ex}

%\subsubsection{Lotka-Volterra three species system}
\begin{ex} \label{ex:LV3} 
By including an additional variable $z$, the three species Lotka-Volterra model is:
\begin{align*}
\dot{x}_1 &= p_1 x_1 - p_2 x_1x_2,\\
\dot{x}_2 &= -p_3 x_2+ p_4 x_1x_2- p_5 x_2x_3,\\
\dot{x}_3 &= - p_6 x_3+ p_7 x_2x_3,
\end{align*}

Assuming only $y=x_1$ is observable. After differential elimination, the differential invariant is:

%%%a = p1, b = p2, c = p3, d=p4, e=p5, f=p6, g=p7

\begin{align*}
& \left(p_1^2p_4 p_6  -\frac{p_1^3p_4 p_7  }{p_2  }\right)y^5 +\left(\frac{p_1^3p_3 p_7  }{p_2  }-p_1^2p_3 p_6  \right)y^4+\left(\frac{3p_1^2p_4 p_7  }{p_2  }+p_1^2p_4 +2p_1p_4 p_6  \right) \dot{y}y^4 + \left(2p_1p_3 p_6  -\frac{3p_1^2p_3 p_7  }{p_2  }\right)\dot{y}y^3\\
 &+\left(p_4 p_6  - 2p_1p_4 -\frac{3p_1p_4 p_7  }{p_2  }\right)\dot{y}^2y^3 +\left(\frac{p_1^2p_7  +3p_1p_3 p_7  }{p_2  }-p_3 p_6  -p_1p_6  \right)\dot{y}^2y^2 
 +\left(\frac{p_4 p_7  }{p_2  }+p_4 \right)\dot{y}^3y^2\\
 &+\left(2p_1+p_6  -\frac{2p_1p_7  +p_3 p_7  }{p_2  }\right)\dot{y}^3y
+\frac{p_7  }{p_2  }\dot{y}^4+\left(p_1p_6  -\frac{p_1^2p_7  }{p_2  }\right)\ddot{y}y^3
 +\left(\frac{2p_1p_7  }{p_2  }-3p_1-p_6  \right)\ddot{y}\dot{y}y^2-\frac{p_7  }{p_2  }\ddot{y}\dot{y}^2y +p_1\dddot{y}y^3\\
&=  -\ddot{y}^2y^2    + \dddot{y}\dot{y}y^2  -\ddot{y}\dot{y}^2y + \dot{y}^4.
\end{align*}
\end{ex}

\begin{ex} \label{ex:Lor} 
Another three species model, the Lorenz model, is described by the system of equations:
%\subsubsection{Lorenz system}
\begin{align*}
\dot{x}_1 &= p_1(x_2-x_1),\\
\dot{x}_2 &= x_1(p_2-x_3)-x_2,\\
\dot{x}_3 &= x_1x_2 - p_3 x_3,
\end{align*}
We assume only $y=x_1$ is observable, perform differential elimination, and obtain the following invariant:
\begin{align*}
- (p_1+p_3) \ddot{y}y + p_1\dot{y}^2 - (p_1 p_3+p_3)\dot{y}y - p_1y^4 + (p_1 p_2 p_3-p_1 p_3)y^2 
= \dddot{y}y - \ddot{y}\dot{y} + \ddot{y}y - \dot{y}^2 + \dot{y}y^3.
\end{align*}

\end{ex}

\begin{ex} \label{ex:LC2} 
%\subsubsection{Lotka-Volterra two species system}
A linear 2-compartment model without input can be written as:
\begin{align*}
\dot{x}_1 &= p_{11}x_1 + p_{12}x_2,\\
\dot{x}_2 &= p_{21}x_1 + p_{22}x_2,
\end{align*}
where $x_1$ and $x_2$ are variables, and $p_{11},p_{12},p_{21},p_{22}$ are parameters. We assume only $x_1$ is  observable and perform differential elimination and obtain our differential invariant in terms of only $y=x_1(t)$:
$$\ddot{y} - (p_{11}+p_{22})\dot{y} + (p_{11}p_{22}-p_{12}p_{21})y = 0$$

\end{ex}

\begin{ex} \label{ex:LC3} 
%\subsubsection{Lotka-Volterra two species system}
The Linear 3-Compartment model without input is:
\begin{align*}
\dot{x}_1 &= p_{11}x_1 + p_{12}x_2 + p_{13}x_3,\\
\dot{x}_2 &= p_{21}x_1 + p_{22}x_2 + p_{23}x_3,\\
\dot{x}_3 &= p_{31}x_1 + p_{32}x_2 + p_{33}x_3,
\end{align*}
where $x_1, x_2, x_3$ are variables, and $p_{11}. p_{12}, p_{13}, p_{21}, p_{22}, p_{23}, p_{31}, p_{32}, p_{33}$ are parameters. We assume only $x_1$ is  observable and perform differential elimination and obtain our differential invariant in terms of only $y=x_1(t)$:
\begin{align*}
& \dddot{y} - (p_{11} + p_{22} + p_{33})\ddot{y} + (p_{12}p_{21} - p_{11}p_{22} + p_{13}p_{31} + p_{23}p_{32} - p_{11}p_{33} - p_{22}p_{33})\ddot{y} \\
& - (-p_{13}p_{22}p_{31} + p_{12}p_{23}p_{31} + p_{13}p_{21}p_{32} - p_{11}p_{23}p_{32} - p_{12}p_{21}p_{33} +
p_{11}p_{22}p_{33})y = 0
\end{align*}

\end{ex}

 \par
By assuming $y=x_1$ in Examples 6.1--6.5 represents the same observable variable, we apply our method to data simulated from each model and perform model comparison. The models are simulated and 100 time points are obtained variable $x$ in each model. We add different levels of Gaussian noise to the simulated data, and then estimate the higher order derivatives from the data. For example, during our study we found that for some parameters of the Lotka-Volterra three species model, e.g. $ p_1,p_2,p_3,p_4,p_5,p_6,p_7=[1.24; 1.68;3.26;0.38;1.50;0.15;1.14]$, we obtained a positive log-likelihood, which meant that we could not estimate the higher order derivatives of the data. Once the data is obtained and derivative data are estimated through the GP regression, each model data set is tested against the other differential invariants. Results are shown in Figure~\ref{fig-four}, where a value of 0, means model rejected, and 1 means model is compatible. We find that we can reject the three species Lotka-Volterra model and Lorenz model for data simulated from the Lotka-Volterra two species; however both linear compartment models are compatible. For data from the three species Lotka-Volterra model, the linear compartment models and two-species Lotka-Volterra can be rejected until the noise increases and then the method can no longer reject any models. Finally data generated from the Lorenz model can only reject the two species linear compartment and two species Lotka-Volterra model.

\begin{figure}[h!]
\centering
\includegraphics[width=\textwidth]{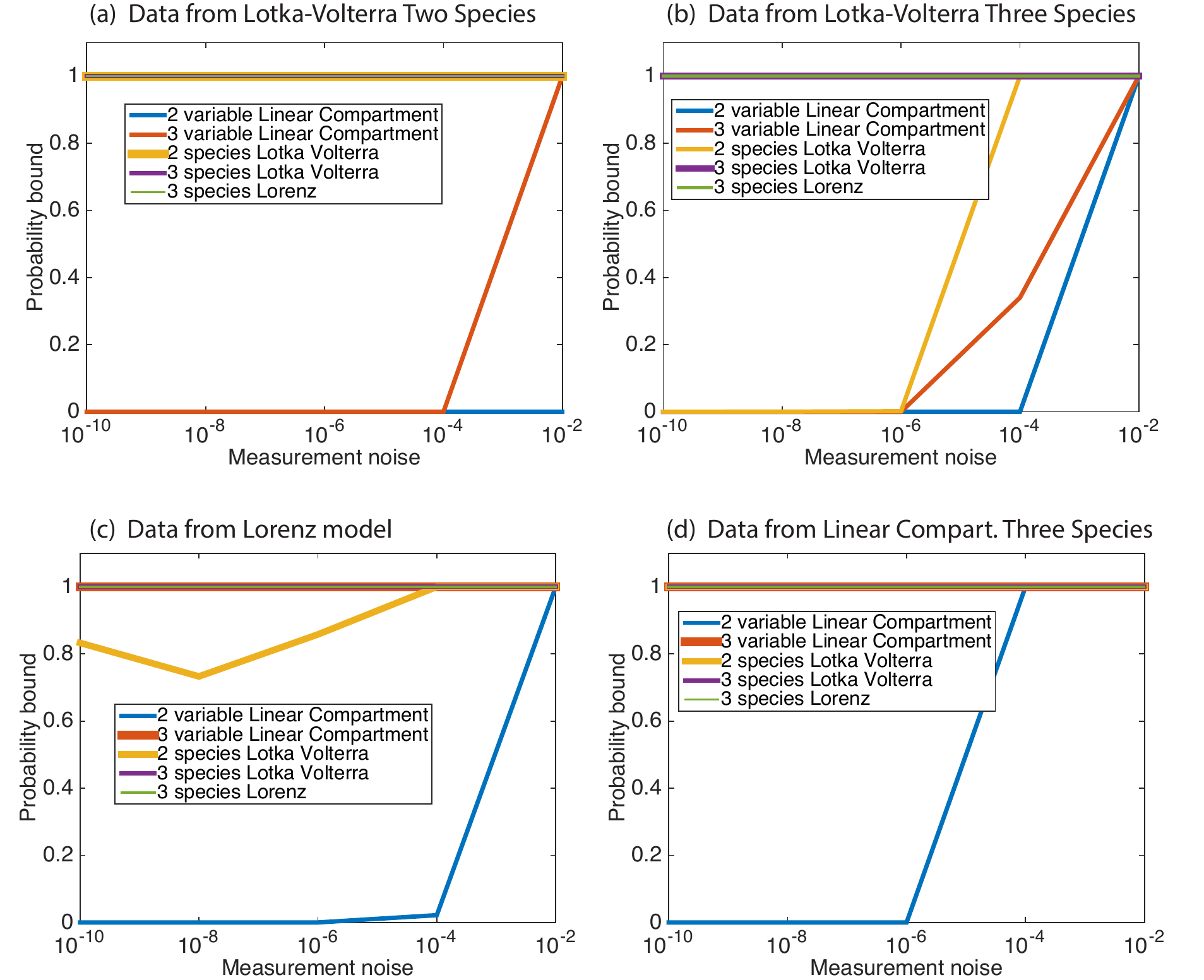}
\caption{Data simulated from model specified and differential algebraic statistics model rejection applied to five model invariants. Gaussian noise is added to data in increments of 0.1 at levels shown in figure.
(a) Data simulated from two species Lotka-Volterra model with parameter values$p_1,p_2,p_3,p_4 =[1.24; 1.68;3.26;0.38]$ and initial condition $z0 = [10,1]$. (b) Data simulated from three species Lotka-Volterra model with parameter values $p_1,p_2,p_3,p_4,p_5,p_6,p_7 =[0.178; 0.12; 0.99;0.17;0.03;0.56;0.88]$ and initial condition $z0 = [2,1,1]$. (c) Data simulated from the Lorenz model with parameter values $p_1,p_2,p_3 =[3.5,.3,2.8]$ and initial condition $z0 = [2,1,1]$.  (d) Data simulated from the Linear Compartment three species model with parameter values $p_{11},p_{12},p_{13},p_{21},p_{22},p_{23},p_{31},p_{32},p_{33} = [-2,1,0,1,-3,1,0,1,-2]$ and initial condition $z0 = [3,1,5]$.  }
\label{fig-four}
\end{figure}

%Lorenz data:
%We can reject the LC2 and LV2 but not the LV3 or LC3.
%Lorenz model simulated with parameters $a,b,c =[3.5,.3,2.8] ,$ Initial condition $z0 = [2,1,1]$
%
%LV3 data:
%Cannot reject anything because parameter values to simulate the data
%were not amenable to our GP regression so a bad negative log
%likelihood value, i.e. positive.
%LV3 model simulated $a,b,c,d,e,f,g =[1.24; 1.68;3.26;0.38;1.50;0.15;1.14]$ , Initial condition $z0 = [2,1,1]$

%If we try with only three time points this isn't a very good parameter set, but we get a solution. As we add more time points, we get the correct parameters back. This was for the Lotka Volterra 2 species model, but there are some things to work through (since our coefficients were dependent how to deal with this), can we say how many more time points do we need for imperfect data vs perfect to get the correct parameters back? As we increase the noise, will we ever get an inconsistent (no solution), ie a false positive? 

\section{Other Considerations: known parameter values, algebraic dependencies, and identifiability}

We have demonstrated our model discrimination algorithm on various models.  In this section, we consider some other theoretical points regarding differential invariants.  

Note that we have assumed that the parameters are all unknown and we have not taken any possible algebraic dependencies among the coefficients into account.  This latter point is another reason our algorithm only concerns model rejection and not model selection.  Thus, each unknown coefficient is essential treated as an independent unknown variable in our linear system of equations.   However, there may be instances where we'd like to consider incorporating this additional information.  We first consider the effect of incorporating known parameter values.

In \cite{Meshkat2015}, an explicit formula for the input-output equations for linear models was derived.  In particular, it was shown that all linear $n-$compartment models corresponding to strongly connected graphs with at least one leak and having the same input and output compartments will have the same differential polynomial form of the input-output equations.  For example, a linear 2-compartment model with a single input and output in the same compartment and corresponding to a strongly connected graph with at least one leak has the form:
$$ \ddot{y}+c_{1}\dot{y}+c_{2}y=\dot{u}+c_{3}u$$

Thus, our model discrimination method would not work for two distinct linear 2-compartment models with the above-mentioned form. In order to discriminate between two such models, we need to take other information into account, e.g. known parameter values. 

\begin{ex} Consider the following two linear 2-compartment models:

$$
\begin{pmatrix} 
\dot{x}_1 \\
\dot{x}_2 \end{pmatrix} = {\begin{pmatrix} 
-p_{01}-p_{21} & p_{12} \\
p_{21} & -p_{12}
\end{pmatrix}} {\begin{pmatrix}
x_1 \\
x_2 \end{pmatrix} } + {\begin{pmatrix}
u \\
0 \end{pmatrix}}, \quad \quad y=x_1$$

$$
\begin{pmatrix} 
\dot{x}_1 \\
\dot{x}_2 \end{pmatrix} = {\begin{pmatrix} 
- p_{21} & p_{12} \\
p_{21} & -p_{02}-p_{12}
\end{pmatrix}} {\begin{pmatrix}
x_1 \\
x_2 \end{pmatrix} } + {\begin{pmatrix}
u \\
0 \end{pmatrix}}, \quad \quad y=x_1$$

whose corresponding input-output equations are of the form:
\begin{align*}
 \ddot{y}+(p_{01}+p_{21}+p_{12})\dot{y}+p_{01}p_{12}y&=\dot{u}+p_{12}u\\
 \ddot{y}+(p_{21}+p_{12}+p_{02})\dot{y}+p_{02}p_{21}y&=\dot{u}+(p_{02}+p_{12})u
 \end{align*}

Notice that both of these equations are of the above-mentioned form, i.e. both 2-compartment models have a single input and output in the same compartment and correspond to strongly connected graphs with at least one leak. In the first model, there is a leak from the first compartment and an exchange between compartments $1$ and $2$.  In the second model, there is a leak from the second compartment and an exchange between compartments $1$ and $2$.  Assume that the parameter $p_{12}$ is known.  In the first model, this changes our invariant to:
$$(p_{01}+p_{21})\dot{y}+p_{01}(p_{12}y)=\dot{u}+p_{12}u-\ddot{y}-p_{12}\dot{y}, \quad\text{or,} \quad
 c_{1}\dot{y}+c_{2}(p_{12}y)=\dot{u}+p_{12}u-\dot{y}-p_{12}\dot{y} $$

In the second model, our invariant is:
$$ (p_{21}+p_{02})\dot{y}+p_{02}p_{21}y-p_{02}u=\dot{u}+p_{12}u-\ddot{y}-p_{12}\dot{y} \quad \text{or,} \quad
 c_{1}\dot{y}+c_{2}y+c_{3}u=\dot{u}+p_{12}u-\ddot{y}-p_{12}\dot{y}$$

In this case, the right-hand sides of the two equations are the same, but the first equation has two variables (coefficients) while the second equation has three variables (coefficients).  Thus, if we had data from the second model, we could try to reject the first model (much like the 3-compartment versus 2-compartment model discrimination in the examples below).  In other words, a vector in the span of $\dot{y},y,$ and $u$ for $t_1,t_2,t_3$ may not be in the span of $\dot{y}$ and $y$ only.

\end{ex}

We next consider the effect of incorporating coefficient dependency relationships.  While we cannot incorporate the polynomial algebraic dependency relationships among the coefficients in our linear algebraic approach to model rejection, we can include certain dependency conditions, such as certain coefficients becoming known constants.  We have already seen one way in which this can happen in the previous example (from known nonzero parameter values).  We now explore the case where certain coefficients go to zero.  From the explicit formula for input-output equations from \cite{Meshkat2015}, we get that a linear model without any leaks has a zero term for the coefficient of $y$.  Thus a linear 2-compartment model with a single input and output in the same compartment and corresponding to a strongly connected graph without any leaks has the form:

$$ \ddot{y}+c_{1}\dot{y}=\dot{u}+c_{2}u$$

Thus to discriminate between two distinct linear 2-compartment models, one with leaks and one without any leaks, we should incorporate this zero coefficient into our invariant.

\begin{ex} Consider the following two linear 2-compartment models:

$$
\begin{pmatrix} 
\dot{x}_1 \\
\dot{x}_2 \end{pmatrix} = {\begin{pmatrix} 
-p_{01}-p_{21} & p_{12} \\
p_{21} & -p_{12}
\end{pmatrix}} {\begin{pmatrix}
x_1 \\
x_2 \end{pmatrix} } + {\begin{pmatrix}
u \\
0 \end{pmatrix}},\quad \quad y=x_1$$

$$
\begin{pmatrix} 
\dot{x}_1 \\
\dot{x}_2 \end{pmatrix} = {\begin{pmatrix} 
- p_{21} & p_{12} \\
p_{21} & -p_{12}
\end{pmatrix}} {\begin{pmatrix}
x_1 \\
x_2 \end{pmatrix} } + {\begin{pmatrix}
u \\
0 \end{pmatrix}}, \quad \quad y=x_1$$

whose corresponding input-output equations are of the form:
\begin{align*}
\ddot{y}+(p_{01}+p_{21}+p_{12})\dot{y}+p_{01}p_{12}y=&\dot{u}+p_{12}u\\
\ddot{y}+(p_{21}+p_{12})\dot{y}=&\dot{u}+p_{12}u
\end{align*}
In the first model, there is a leak from the first compartment and an exchange between compartments $1$ and $2$.  In the second model, there is an exchange between compartments $1$ and $2$ and no leaks.  Thus, our invariants can be written as:
\begin{align*}
 c_{1}\dot{y}+c_{2}y+c_{3}u=\dot{u}-\ddot{y}\\
 c_{1}\dot{y}+c_{2}u=\dot{u}-\ddot{y}
\end{align*}

Again, the right-hand sides of the two equations are the same, but the first equation has three variables (coefficients) while the second equation has two variables (coefficients).  Thus, if we had data from the first model, we could try to reject the second model.  In other words, a vector in the span of $\dot{y},y,$ and $u$ for $t_1,t_2,t_3$ may not be in the span of $\dot{y}$ and $u$ only.

\end{ex}

Finally, we consider the identifiability properties of our models.  If the number of parameters is greater than the number of coefficients, then the model is unidentifiable.  On the other hand, if the number of parameters is less than or equal to the number of coefficients, then the model could possibly be identifiable.  Clearly, an identifiable model is preferred over an unidentifiable model.  We note that, in our approach of forming the linear system $A \kappa = b$ from the input-output equations, we could in theory solve for the coefficients $\kappa$ and then solve for the parameters from these known coefficient values if the model is identifiable \cite{Boulier}.  However, this is not a commonly used method to estimate parameter values in practice.
% and the reader is directed to more standard methods in (cite some parameter estimation papers).  

As noted above, the possible algebraic dependency relationships among the coefficients are not taken into account in our linear algebra approach.  This means that there could be many different models with the same differential polynomial form of the input-output equations.  If such a model cannot be rejected, we note that an identifiable model satisfying a particular input-output relationship is preferred over an unidentifiable one satisying the same form of the input-output relations, as we see in the following example.       

\begin{ex} Consider the following two linear 2-compartment models:

$$
\begin{pmatrix} 
\dot{x}_1 \\
\dot{x}_2 \end{pmatrix} = {\begin{pmatrix} 
-p_{01}-p_{21} & p_{12} \\
p_{21} & -p_{12}
\end{pmatrix}} {\begin{pmatrix}
x_1 \\
x_2 \end{pmatrix} } + {\begin{pmatrix}
u \\
0 \end{pmatrix}}, \quad \quad y=x_1$$

$$
\begin{pmatrix} 
\dot{x}_1 \\
\dot{x}_2 \end{pmatrix} = {\begin{pmatrix} 
-p_{01}-p_{21} & p_{12} \\
p_{21} & -p_{02}-p_{12}
\end{pmatrix}} {\begin{pmatrix}
x_1 \\
x_2 \end{pmatrix} } + {\begin{pmatrix}
u \\
0 \end{pmatrix}}, \quad \quad y=x_1$$

whose corresponding input-output equations are of the form:
\begin{align*} 
\ddot{y}+(p_{01}+p_{21}+p_{12})\dot{y}+p_{01}p_{12}y=&\dot{u}+p_{12}u\\
 \ddot{y}+(p_{01}+p_{21}+p_{12}+p_{02})\dot{y}+(p_{01}p_{02}+p_{01}p_{12}+p_{02}p_{21})y=&\dot{u}+(p_{02}+p_{12})u
 \end{align*}

In the first model, there is a leak from the first compartment and an exchange between compartments $1$ and $2$.  In the second model, there are leaks from both compartments and an exchange between compartments $1$ and $2$.  Thus, both models have invariants of the form:
$$ \ddot{y}+c_{1}\dot{y}+c_{2}y=\dot{u}+c_{3}u$$

Since the first model is identifiable and the second model is unidentifiable, we prefer to use the form of the first model if the model's invariant cannot be rejected.
\end{ex}

\section{Discussion/Conclusion}
After performing this differential algebraic statistics model rejection, one has already obtained the input-output equations to test structural identifiability \cite{Ljung, Ollivier, Saccomani}. In a sense, our method extends the current spectrum of potential approaches for comparing models with time course data, in that one first can reject incompatible models, then test structural identifiability of compatible models using input-output equations obtained from the differential elimination, infer parameter values of the admissible models, and apply an information criterion model selection method to assert the best model. 

Notably the presented differential algebraic statistics method does not penalize for model complexity, unlike traditional model selection techniques. Rather, we reject when a model cannot, for any parameter values, be compatible with the given data. We found that simpler models, such as the linear 2 compartment model could be rejected when data was generated from a more complex model, such as the three species Lotka-Volterra model, which elicits a wider range of behavior. On the other hand, more complex models, such as the Lorenz model, were often not rejected, from data simulated from less complex models. In future it would be helpful to better understand the relationship between differential invariants and dynamics. We also think it would be beneficial to investigate algebraic properties of sloppiness \cite{gutenkunst:2007:plos-comput-biol}.

We believe there is large scope for additional parameter-free coplanarity model comparison methods. It would be beneficial to explore which algorithms for differential elimination can handle larger systems, and whether this area could be extended.

\section*{Acknowledgments}
The authors acknowledge funding from the American Institute of Mathematics (AIM) where this research commenced. The authors thank Mauricio Barahona, Mike Osborne, and Seth Sullivant for helpful discussions. We are especially grateful to Paul Kirk for discussions on GPs and providing his GP code, which served as an initial template to get started. NM was partially supported by the David and Lucille Packard Foundation. HAH acknowledges funding from AMS Simons Travel Grant, EPSRC Fellowship EP/K041096/1 and MPH Stumpf Leverhulme Trust Grant.

\end{document}